\documentclass[12pt]{amsproc}
  \usepackage{latexsym} 
  \usepackage[all]{xy}
  \usepackage{amsfonts} 
  \usepackage{amsthm} 
  \usepackage{amsmath} 
  \usepackage{amssymb}
  \usepackage{pifont}  
  \usepackage{enumerate}
  \xyoption{2cell}

  \def\<{{\langle}} 
  \def\>{{\rangle}}

  \def\note#1{{}}

  \def\note#1{} 

   \def\cK{{\mathcal K}}  
  \def\cO{{\mathcal O}}

  \def\beq{\begin{equation}} 
  \def\eeq{\end{equation}}

  \def\id{\mathrm{id}} 
  \def\im{{\rm Im}}

  \def\ot{{\otimes}}

 \def\coker{\mathrm{coker}}

  \newcounter{zlist}

  \newcounter{blist}

  \newcounter{rlist}

\def\stac#1{\raise-.2cm\hbox{$\stackrel{\displaystyle\otimes}{\scriptscriptstyle{#1}}$}}

\def\cten#1{\raise-.2cm\hbox{$\stackrel{\displaystyle\widehat{\otimes}}
{\scriptscriptstyle{#1}}$}}

  \def\Label#1{\label{#1}\ifmmode\llap{[#1] }\else 
  \marginpar{\smash{\hbox{\tiny [#1]}}}\fi} 
  \def\Label{\label}

  \newtheorem{proposition}{Proposition}[section]

  \theoremstyle{definition}

  \theoremstyle{remark}

  \newcounter{c} 
   
  \newcommand{\etyk}[1]{\vspace{-7.4mm}$$\begin{equation}\Label{#1} 
  \addtocounter{c}{1}} 
  \renewcommand{\]}{\ifnum \value{c}=1 $$\else \end{equation}\fi} 
\def\ot{\otimes}

\def\CC{{\mathbb C}}

\def\NN{{\mathbb N}}

\def\PP{{\mathbb P}}

\def\RR{{\mathbb R}}

\def\WW{{\mathbb W}}

\def\ZZ{{\mathbb Z}}

\newcommand{\Cc}{\mathcal{C}}

\newcommand{\Hh}{\mathcal{H}}

\newcommand{\Kk}{\mathcal{K}}

\newcommand{\Tt}{\mathcal{T}}

\def\*C{{}^*\hspace*{-1pt}{\Cc}}

\def\text#1{{\rm {\rm #1}}}

 \def\1{\mathbf{1}}

\begin{document} 

\title{Circle actions on a quantum Seifert manifold} 
 \author{Tomasz Brzezi\'nski}
 \address{ Department of Mathematics, Swansea University, 
  Singleton Park, \newline\indent  Swansea SA2 8PP, U.K.} 
  \email{T.Brzezinski@swansea.ac.uk}

  \date{March 2012} 
  \begin{abstract} 
The quotients of a (non-orientable) quantum Seifert manifold by circle actions are described. In this way quantum weighted real projective spaces that include the quantum disc and the quantum real projective space as special cases are obtained. Bounded irreducible representations of the coordinate algebras and the $K$-groups of the algebras of continuous functions on quantum weighted real projective spaces are presented.
  \end{abstract} 
  \maketitle

\section{Introduction}
It is well known that 2-dimensional orbifolds or Riemann orbifold surfaces can be obtained as quotients of Seifert 3-manifolds by circle actions; see \cite[Example~1.16]{AdeLei:orb} or \cite{Sco:geo}. All such actions have been classified in \cite{Ray:cla}. In \cite{BrzFai:tea} we studied circle actions on the quantum 3-sphere (and other quantum odd-dimensional spheres) and  obtained quantum weighted projective spaces as the quotients  of the quantum 3-sphere by these actions (fixed points of the corresponding coactions). These can be considered as explicit examples of quantum orbifolds.

On the other hand, an algebraic description of a (non-orientable) quantum Seifert 3-manifold was presented in \cite{BrzZie:pri}. This quantum space was obtained by the prolongation of the $\ZZ_2$-action on the quantum 2-sphere to the circle action. The aim of this note is to describe circle actions on this quantum Seifert 3-manifold, identify their quotients (fixed points of the corresponding coactions), analyse their representations and calculate topological $K$-groups. In this way we demonstrate that methods employed  by Hajac, Matthes and Szyma\'nski  \cite{HajMat:rea} in the analysis of the structure of the quantum disc and the quantum real projective space as well as the results obtained threin can be transferred to this new class of quantum spaces.

Throughout we work with involutive algebras over the field of complex numbers (but the algebraic results remain true for all fields of characteristic 0). All algebras are associative and have identity.

\section{Coordinate algebras of quantum real weighted projective spaces}\label{sec.alg} 
\setcounter{equation}{0}
Coordinate algebras of quantum real projective spaces $\cO(\RR\PP_q^{2n})$ are defined as fixed points of the $\ZZ_2$ action or $\cO(\ZZ_2)$-coaction on the coordinate algebras of quantum spheres $\cO(S_q^{2n})$. In this way the quantum real projective space becomes a base of the $\ZZ_2$ principal bundle with the quantum sphere as the total space. The structure group of this bundle can be prolonged to $U(1)$. The coordinate $*$-algebras $\cO(\Sigma^{2n+1}_q)$ of the resulting total space are  generated by $\zeta_0, \zeta_1,\ldots, \zeta_n$ and a central unitary $\xi$ subject to the following relations
$$
\label{kle1}
\zeta_i\zeta_j = q\zeta_j\zeta_i \quad \mbox{for $i<j$}, \qquad \zeta_i\zeta^*_j = q\zeta_j^*\zeta_i \quad \mbox{for $i\neq j$},
$$
\begin{equation}\label{kle}
\zeta_i\zeta_i^* = \zeta_i^*\zeta_i + (q^{-2}-1)\sum_{m=i+1}^n \zeta_m\zeta_m^*, \qquad \sum_{m=0}^n \zeta_m\zeta_m^*=1, \qquad \zeta_n^* = \zeta_n\xi,
\end{equation}
where $q\in (0,1)$ is a deformation parameter; see  \cite[Proposition~5.2]{BrzZie:pri}.

For any choice of $n+1$ pairwise coprime numbers $l_0,\ldots, l_n$,  one can define the coaction of the Hopf algebra $\cO(U(1)) = \CC [u,u^*]$, where $u$ is a unitary and grouplike generator, on $\cO(\Sigma^{2n+1}_q)$ as
\begin{equation}\label{coaction}
\varrho_{l_0,\ldots, l_n} : \cO(\Sigma^{2n+1}_q) \to \cO(\Sigma^{2n+1}_q) \ot \CC [u,u^*] ,\qquad \xi \mapsto \xi\ot u^{-2l_n}, \quad \zeta_i \mapsto z_i\ot u^{l_i},
\end{equation}
$i=0,1,\ldots , n.$ 
This coaction is then extended to the whole of $\cO(\Sigma^{2n+1}_q)$ so that $\cO(\Sigma^{2n+1}_q)$ is a right $\CC [u,u^*]$-comodule algebra. The algebra 
of coordinate functions on the quantum real weighted projective space is now defined as the subalgebra of $\cO(\Sigma^{2n+1}_q)$ containing all elements invariant under the coaction $\varrho_{l_0,\ldots, l_n}$, i.e.\
$$
\cO(\RR\PP_q(l_0,l_1,\ldots, l_n)) = {\cO(\Sigma^{2n+1}_q)}^{co\CC [u,u^*]}:= \{x\in\cO(\Sigma^{2n+1}_q) \; |\; \varrho_{l_0,\ldots, l_n}(x) = x\ot 1\}.
$$
In the case $l_0=l_1 = \cdots = 1$ one obtains the algebra of functions on the quantum real  projective space $\RR\PP_q^{2n}$. 

In this note we are concerned with the two-dimensional case $n=1$. $\cO(\Sigma^{3}_q)$ is thus generated by $\zeta_0, \zeta_1$ and the central unitary $\xi$ such that
\begin{equation}
\label{kle.dim2}
\zeta_0\zeta_1 = q\zeta_1\zeta_0,
\qquad
\zeta_0\zeta_0^* = \zeta_0^*\zeta_0 + (q^{-2}-1)\zeta_1^2\xi, \qquad  \zeta_0\zeta_0^*+ \zeta_1^2\xi=1, \qquad   \zeta_1^* = \zeta_1\xi .
\end{equation}
For two coprime integers $k,l$, the $*$-algebra coaction $\varrho:= \varrho_{k,l}$ of $\cO(U(1))$ on $\cO(\Sigma)^2$ comes out as
\begin{equation}
\label{coa}
 \varrho(\zeta_0) = \zeta_0\ot u^k, \qquad \varrho(\zeta_1) = \zeta_1\ot u^l, \qquad \varrho(\xi) = \xi\ot u^{-2l},
 \end{equation}
 Note that, in view of the last of equations \eqref{kle.dim2} the value of $\varrho$ on $\xi$ is determined by its value on $\zeta_1$ (and vice versa), since $\varrho$ is required to be a $*$-algebra map. Presently we determine the structure of the coinvariant subalgebra of $\cO(\Sigma^{3}_q)$.
 
 As explained in \cite{BrzZie:pri}, and as it is evident from relations \eqref{kle.dim2}, the linear basis for $\cO(\Sigma^{3}_q)$ is provided by $1$ and 
 $$
 \zeta_0^m \zeta_1^p\xi^r, \qquad {\zeta_0^*}^m \zeta_1^p\xi^r, \qquad m,p\in \NN, \; r\in \ZZ, \qquad (m,p,r)\neq (0,0,0).
 $$
 Obviously, the second set of basis elements is (proportional to) the $*$-conjugate of the first, hence we can restrict the analysis of coinvariants to the first set and then take the $*$-conjugate of the results. Since $\varrho$ is an algebra map,
 $$
 \varrho:  \zeta_0^m \zeta_1^p\xi^r\mapsto \zeta_0^m \zeta_1^p\xi^r\ot u^{km + (p-2r)l}.
 $$
 Hence $\zeta_0^m \zeta_1^p\xi^r$ is a coinvariant element if and only if $km + (p-2r)l=0$. Since $k$ and $l$ are coprime, there must exist $n\in \NN$ such that $m=nl$ and hence   $p = 2r-kn$. So, all elements of $\cO(\Sigma^{3}_q)$ invariant under the coaction $\varrho$ must be linear combinations of
 $$
 \zeta_0^{nl} \zeta_1^{2r-kn}\xi^r
 $$
 and their $*$-conjugates. At this point we need to consider two cases.
 
 First, if $k=2s$ is an even number, then $l$ is odd, $p$ is an even number,  $r-sn \in \NN$ and 
 $$
 \zeta_0^{nl} \zeta_1^{2r-kn}\xi^r= \zeta_0^{nl} \zeta_1^{2(r-sn)}\xi^r = (\zeta_0^{l}\xi^s)^n (\zeta_1^{2}\xi)^{r-sn}. 
 $$
The elements $a:= \zeta_1^{2}\xi$ and $c_+:=\zeta_0^{l}\xi^s$ are invariant under the coaction $\varrho$ and they generate the $*$-subalgebra of coinvariants of $\cO(\Sigma^{3}_q)$. Note that $a^* = a$. 

Second, if $k=2s-1$ is an odd number, then set $t$ to be the least natural number such that $r - sn +t \geq 0$. Since $p = 2r-kn = 2r - 2sn +n$ is non-negative also $n-2t$ is non-negative and we can write
$$
 \zeta_0^{nl} \zeta_1^{2r-kn}\xi^r = \zeta_0^{nl} \zeta_1^{2(r-sn) +2t + n-2t}\xi^r
  \sim  (\zeta_0^{l}\zeta_1 \xi^s)^{n-2t} (\zeta_1^{2}\xi)^{r-sn +t} (\zeta_0^{2l}\xi^k)^t . 
$$
The elements $a:= \zeta_1^{2}\xi$, $b = \zeta_0^{l}\zeta_1 \xi^s$ and $c_-:=\zeta_0^{2l}\xi^k$ are invariant under the coaction $\varrho$ and they generate the $*$-subalgebra of coinvariants of $\cO(\Sigma^{3}_q)$.  

\begin{proposition}\label{prop.algebras}
(1) For $k=2s$ and all $l$ coprime with $k$, the $\varrho_{k,l}$-coinvariant $*$-subalgebra of $\cO(\Sigma^{3}_q)$ is generated by $a =\zeta_1^{2}\xi$ and $c_+ = \zeta_0^{l}\xi^s$, which satisfy the following relations:
\begin{equation}\label{even}
a^* =a, \qquad ac_+ = q^{-2l}c_+a, \qquad c_+c_+^* = \prod_{m=0}^{l-1}(1-q^{2m}a), \qquad c_+^*c_+ =  \prod_{m=1}^l(1-q^{-2m}a).
\end{equation}
As these relations do not depend on the choice of $k$ but only on its parity, we can choose $k=2$ and $l$ any odd number. A $*$-algebra generated by $a, c_+$ subject to relations \eqref{even} will be denoted by $\cO(\RR\PP_q^2(l;+))$. 

(2) For $k=2s-1$ and all $l$ coprime with $k$, the $\varrho_{k,l}$-coinvariant $*$-subalgebra of $\cO(\Sigma^{3}_q)$ is generated by $a =\zeta_1^{2}\xi$, $b= \zeta_0^{l}\zeta_1 \xi^s$ and $c_- = \zeta_0^{2l}\xi^k$, which satisfy the following relations:
$$
a^* =a, \qquad ab = q^{-2l}ba, \qquad ac_- = q^{-4l}c_-a, \qquad b^2 = q^{3l}ac_-, \qquad bc_- = q^{-2l}c_-b, 
$$
\begin{equation}\label{odd}
bb^* = q^{2l}a\prod_{m=0}^{l-1}(1-q^{2m}a), \quad b^*b = a \prod_{m=1}^l(1-q^{-2m}a), \quad b^*c_- = q^{-l} \prod_{m=1}^l(1-q^{-2m}a)b, 
\end{equation}
$$
c_-b^* = q^{l}b\prod_{m=0}^{l-1}(1-q^{2m}a), \qquad c_-c_-^* = \prod_{m=0}^{2l-1}(1-q^{2m}a), \qquad c_-^*c_- =  \prod_{m=1}^{2l}(1-q^{-2m}a). 
$$
As these relations do not depend on the value of $k$ but only on its parity, we can choose $k=1$ and $l$ any positive integer. A $*$-algebra generated by $a, b, c_-$ subject to relations \eqref{even} will be denoted by $\cO(\RR\PP_q^2(l;-))$. 
\end{proposition}

\begin{proof}
Applying relations \eqref{kle.dim2} repetitively one easily finds that
\begin{eqnarray}
\zeta_0^m {\zeta_0^*}^n &=&
\begin{cases}
\prod_{p=0}^{m-1} \left(1-q^{2p}\zeta_1^2 \xi \right), & \mbox{if \ $m=n$,}\\
\zeta_0^{m-n}\prod_{p=0}^{n-1} \left(1-q^{2p}\zeta_1^2 \xi \right), & \mbox{if \ $m > n$,}\\
\prod_{p=0}^{m-1} \left(1-q^{2p}\zeta_1^2 \xi \right) {\zeta_0^*}^{n-m},  & \mbox{if \ $n > m$,}
\end{cases} 
\notag 
\\  ~ && ~ \label{powers} \\
{\zeta_0^*}^n \zeta_0^m &=&
\begin{cases}
\prod_{p=1}^{m} \left(1-q^{-2p}\zeta_1^2 \xi\right), &\mbox{if \ $m=n$,}\\
{\zeta_0^*}^{n-m} \prod_{p=1}^{m} \left(1-q^{-2p}\zeta_1^2 \xi\right), & \mbox{if \ $n > m$,}\\
\prod_{p=1}^{n} \left(1-q^{-2p}\zeta_1^2 \xi\right) \zeta_0^{m-n},  & \mbox{if \ $n < m$.}
\end{cases} \notag
\end{eqnarray}
Relations \eqref{even} and \eqref{odd} follow by the defintions of $a,b,c_\pm$ and  equations \eqref{powers} and  \eqref{kle.dim2}. The remaining assertions are consequences of the discussion preceding Proposition~\ref{prop.algebras}.
\end{proof}

At this point we would like to make a few comments on algebras described in Proposition~\ref{prop.algebras}. We begin by looking at the case $l=1$. In $\cO(\RR\PP_q^2(1;+))$, the third relation in \eqref{even} implies that $a$ can be expressed as a product of $c_+$ and $c_+^*$, and hence the whole information about this algebra is contained in the second of relations \eqref{even}. This is a defining relation of the coordinate algebra of the quantum disc \cite{KliLes:two}, hence $\cO(\RR\PP_q^2(1;+))$ is isomorphic to the quantum disc. In the odd case, $\cO(\RR\PP_q^2(1;-))$ is isomorphic to the coordinate algebra of the quantum real projective space \cite{Haj:str}. This is not surprising, since $\varrho_{1,1}$ is the $\cO(U(1))$-coaction which makes $\cO(\Sigma^{3}_q)$ into a noncommutative principal bundle over the quantum real projective space. For any $l$, the subalgebra of $\cO(\RR\PP_q^2(l;-))$ generated by $a$ and $b$ is isomorphic to the coordinate algebra of the quantum teardrop $\WW\PP_q(1,l)$ \cite{BrzFai:tea}. Finally, in the classical limit $q=1$, for all $l>1$ ($l$ odd in the case $+$) the algebras $\cO(\RR\PP_q^2(l;+))$, $\cO(\RR\PP_q^2(l;-))$ describe singular spaces due to the repeated roots of the polynomials on the right hands of equations \eqref{even}, \eqref{odd}. If $q\neq 1$, the roots are separated, so the singularities are  resolved. This is similar to quantum teardrops $\WW\PP_q(1,l)$, and it is another clear indication how noncommutativity can be used to resolve singularities of affine spaces \cite{Van:cre}.

\section{Representations}\label{sec.rep} 
\setcounter{equation}{0}
Bounded $*$-representations of algebras  $\cO(\RR\PP_q^2(l;+))$ and  $\cO(\RR\PP_q^2(l;-))$ are derived and classified in a standard way; see, e.g.\ the proofs of \cite[Theorem~4.5]{HajMat:rea} or \cite[Proposition~2.2]{BrzFai:tea}. Hence we merely list all bounded irreducible $*$-representations (up to a unitary equivalence), and we make a few comments about their relation to representations of $\cO(\Sigma^{3}_q)$.

\subsection{Bounded representations of $\cO(\RR\PP_q^2(l;+))$.}\label{sec.rep.e} There is a family of one-di\-men\-sional representations of $\cO(\RR\PP_q^2(l;+))$ labelled by $\theta\in [0,1)$ and given by
\begin{equation}\label{rep.e}
\pi_\theta(a) = 0, \qquad \pi_\theta (c_+) = e^{2\pi i\theta}.
\end{equation}
All other representations are infinite dimensional, labelled by $r=1,\ldots , l$, and given by
\begin{equation}\label{reps.e}
\pi_r(a) e_n^r = q^{2(ln+r)} e_n^r, \quad \pi_r(c_+) e_n^r =  \prod_{m=1}^{l}\left(1 - q^{2(ln+r-m)}\right)^{1/2}e_{n-1}^r, \quad \pi_r(c_+) e_0^r =0,
\end{equation}
where $e_n^r$, $n\in \NN$ is an orthonormal basis for the representation space $\Hh_r\cong l^2(\NN)$.

\subsection{Bounded representations of $\cO(\RR\PP_q^2(l;-))$.}\label{sec.rep.o}
There is a family of one-di\-men\-sional representations of $\cO(\RR\PP_q^2(l;-))$ labelled by $\theta\in [0,1)$ and given by
\begin{equation}\label{rep.o}
\pi_\theta(a) = 0, \qquad \pi_\theta(b) = 0, \qquad  \pi_\theta(c_-) = e^{2\pi i\theta}.
\end{equation}
All other representations are infinite dimensional, labelled by $r=1,\ldots , l$, and given by
\begin{equation}\label{reps.o.1}
\pi_r(a) e_n^r = q^{2(ln+r)} e_n^r, \quad \pi_r(b) e_n^r =  q^{ln+r}\prod_{m=1}^{l}\left(1 - q^{2(ln+r-m)}\right)^{1/2}e_{n-1}^r, \quad \pi_r(b) e_0^r =0,
\end{equation}
\begin{equation} \label{reps.o}
\pi_r(c_-) e_n^r =  \prod_{m=1}^{2l}\left(1 - q^{2(ln+r-m)}\right)^{1/2}e_{n-2}^r, \qquad \pi_r(c_-) e_0^r = \pi_r(c_-) e_1^r= 0, 
\end{equation}
where $e_n^r$, $n\in \NN$, is an orthonormal basis for the representation space $\Hh_r \cong l^2(\NN)$.

\subsection{Relation to representations of $\cO(\Sigma^{3}_q)$.}\label{sec.rep.rel} Amongst the representations of $\cO(\Sigma^{3}_q)$ listed in \cite{BrzZie:pri} we find an infinite-dimensional representation $\pi_{0,+}$, defined as follows. Let $e_n$ be the infinite orthonormal basis of the representation space $\Hh\cong l^2(\NN)$. Then
\begin{equation}\label{rep.sei}
\pi_{0,+}(\zeta_0^m\zeta_1^p\xi^s) e_n = q^{p(n+1)}\prod_{t=0}^{m-1}\left(1 - q^{2(n-t)}\right)^{1/2}e_{n-m},
\end{equation}
and zero, whenever $m$ exceeds $n$.  This is a faithful representation, since the condition
$$
\pi_{0,+}(\sum_{m,s}\sum_{p=0}^N\mu_{m,p,s}\zeta_0^m\zeta_1^p\xi^s) =0,
$$
 for all $n \in \NN$, yields an (infinite) system of equations
$$
\sum_{p=0}^N(\sum_s\mu_{m,p,s}) q^{p(n+1)} =0, \qquad  n\in \NN.
$$
The determinant of the first $N+1$ of these equations is the Vandermonde determinant, hence it is non-zero. 

For infinite dimensional representation spaces $\Hh_r$ described in Sections~\ref{sec.rep.e} and \ref{sec.rep.o}, consider the vector space isomorphism,
$$
\phi : \bigoplus_{r=1}^{l}\Hh_r \to \Hh, \qquad e_n^r\mapsto e_{ln+r-1}.
$$
 Using \eqref{rep.sei} and \eqref{reps.e} or \eqref{reps.o} one easily finds that, for all $x\in \cO(\RR\PP_q^2(l;\pm))$,
$$
{\phi}\left(\pi_r(x)e_n^r\right) = \pi_{0,+}\left(j_\pm(x)\right) {\phi}(e_n^r),
$$
where the $j_\pm$ are embeddings of $\cO(\RR\PP_q^2(l;\pm))$ into $\cO(\Sigma^{3}_q)$ described in Proposition~\ref{prop.algebras}. Therefore, the direct sum of representations $\pi_r$ of $\cO(\RR\PP_q^2(l;\pm))$ can be seen as the restriction of $\pi_{0,+}$ to the relevant subalgebras of $\cO(\Sigma^{3}_q)$. Since $\pi_{0,+}$ is a faithful representation, the direct sum $\pi$ of all the $\pi_r$ is a faithful representation too.

\section{$K$-theory}\label{sec.top} 
\setcounter{equation}{0}
The $C^*$-algebras $C(\RR\PP_q(l;\pm))$ of continuous functions on the quantum real weighted  projective spaces are  defined as the subalgebras of bounded operators on the Hilbert space $\bigoplus_{r=1}^{l}\Hh_r$ obtained as the completions of $\bigoplus_{r=1}^{l} \pi_r\left( \cO(\RR\PP_q(l;\pm))\right)$. In this section we calculate the $K$-groups for these algebras, closely following the approach of \cite[Section~4]{HajMat:rea}. We also identify the algebras $C(\RR\PP_q(l;\pm))$ as pullbacks of quantum discs or  quantum real projective spaces over the circle algebra.

\subsection{The six-term exact sequence.}\label{sec.six}
Consider closed $*$-ideals $J_\pm$ of $C(\RR\PP_q(l;\pm))$ generated by $a$, and let  $p_\pm: C(\RR\PP_q(l;\pm)) \to  C(\RR\PP_q(l;\pm))/J_\pm$ be the canonical surjections. In view of  relations \eqref{even} and \eqref{odd}, $C(\RR\PP_q(l;\pm))/J_\pm$ is generated by $p_\pm(c_\pm)$ and $p_\pm (c_\pm c_\pm^*) = p_\pm (c_\pm^* c_\pm) =1$, hence $C(\RR\PP_q(l;\pm))/J_\pm\cong C(S^1)$. 

To identify $J_\pm$ we look at the direct sum $\pi = \bigoplus_{r=1}^l \pi_r$ of infinite dimensional  representations $\pi_r$ and we follow the arguments of the proof of \cite[Proposition~1.2]{She:poi}. First, in view of \eqref{reps.e} and \eqref{reps.o.1}, $\pi_r(a)$ is a compact operator (remember that $q\in (0,1)$). Hence $\pi_r(J_\pm) \subseteq \Kk_r$, where $\Kk_r = \Kk(\Hh_r)$ denotes the ideal of compact operators on the representation space $\Hh_r$. Consequently $\pi(J_\pm) \subseteq  \bigoplus_{r=1}^{l}\cK_r$. Since the spectrum of  $\pi(a)$ consists of distinct numbers 0, $q^{2(n+1)}$, $n\in \NN$, 
$J_\pm$ contains all orthogonal projections  to one-dimensional spaces spanned by the $e_n^r$.  Furthermore, $\pi_r(ac_+)$ and $\pi_r(ab)$ are step-by-one operators on their respective Hilbert spaces $\Hh_r$ with non-zero weights, hence $\bigoplus_{r=1}^{l}\cK_r\subseteq\pi(J_\pm)$. In this way we obtain  short exact sequences (one for the $+$ case and one for the $-$ case):
 \begin{equation}\label{seq}
\xymatrix{0 \ar[r] & \bigoplus_{r=1}^{l}\cK_r \ar[r]^-{i_\pm} & C(\RR\PP_q(l;\pm)) \ar[r]^-{p_\pm} & C(S^1) \ar[r] & 0.}
\end{equation}
Since $K_0(C(S^1)) = K_1(C(S^1)) = K_0(\cK_r) = \ZZ$ and $K_1(\cK_r) =0$, the six-term sequences of the $K$-groups lead to the following exact sequence
 \begin{equation}\label{seq.6.red}
\xymatrix{0 \ar[r] &  K_1(C(\RR\PP_q(l;\pm))) \ar[r]^-{} &\ZZ  
\ar[r]^-{\delta_\pm} & \ZZ^l \ar[r]^-{K_0({i}_\pm)} & K_0(C(\RR\PP_q(l;\pm))) \ar[r]^-{K_0({p}_\pm)} & \ZZ \ar[r] & 0.}
\end{equation}
The calculation of the connecting index maps $\delta_\pm$, which can be performed using the method described in \cite[Section~4]{HajMat:rea}, gives distinct answers for the $+$ and $-$ cases, and we record these answers presently.

\subsection{$K$-groups of $C(\RR\PP_q(l;+))$.}\label{sec.k.even}
Since $p_+(c_+)$ is a unitary generator of $C(S^1)$,  its $K_1$-class is $[p_+(c_+)]_1 =1$, and it generates $K_1(C(S^1)) = \ZZ$. To calculate the value of the index map $\delta_+$ on $[p_+(c_+)]_1$ we need to lift $p_+(c_+)$ to a coisometry $d_+ \in C(\RR\PP_q(l;+))$ such that $p_+(c_+) = p_+(d_+)$. Then $\delta_+([p_+(c_+)]_1) = [1 - d_+^*d_+]_0$; see \cite[Section~9.2]{RorLar:int}. Such a coisometry $d_+$ is determined from the condition $\pi_r(d_+) = U_r$, $r=1,2,\ldots, r$, where $U_r$ is the unilateral shift $U_r(e_0^r) = 0$, $ U_r(e_n^r) =e_{n-1}^r$, $n\geq 0$. We write $U = \oplus_{r=1}^l U_r$. Note  that $d_+$ is a coisometry, since clearly $UU^* = 1$, so $\pi (d_+ d_+^*) = \pi(1)$, i.e.\ $d_+ d_+^* =1$, as $\pi$ is a faithful representation.

In view of \eqref{reps.e}, $\pi(c_+) = U \pi (|c_+|)$. Since $\pi$ is a faithful representation, this implies that $c_+ = d_+ |c_+|$. Furthermore, equations \eqref{even} and the definition of $J_+$ imply that $p_+(|c_+|) =1$. Hence
$$
p_+(c_+) = p_+(d_+ |c_+|) =  p_+(d_+),
$$
as required. The existence of $d_+$ is guaranteed by the functional calculus and the observation that 
$$
U_r = \pi_r(c_+)\prod_{m=1}^l (1-q^{-2m} \pi_r(a))^{-{1}/2}.
$$
Since $(1-U_r^*U_r)e_n^r = \delta_{n,0}$, the index map $\delta_+$ assigns $(1,1,\ldots, 1) \in \ZZ^l$ to the class of $p_+(c_+)$, hence
$$
\delta_+ :\ZZ\to \ZZ^l , \qquad m\mapsto (m,m,\ldots, m).
$$
In particular $\delta_+$ is injective and the exactness at the  third term in the sequence \eqref{seq.6.red} yields $K_1(C(\RR\PP_q(l;+))) = 0$. The remainder of the sequence \eqref{seq.6.red} is
\begin{equation} \label{seq.6.rem}
\xymatrix{0 \ar[r]  &\ZZ  
\ar[r]^-{\delta_+} & \ZZ^l \ar[r]^-{K_0({i}_+)} & K_0(C(\RR\PP_q(l;+))) \ar[r]^-{K_0({p}_+)} & \ZZ \ar[r] & 0.}
\end{equation}
The sequence \eqref{seq.6.rem} splits, since $\ZZ$ is a free $\ZZ$-module, hence
$$
K_0(C(\RR\PP_q(l;+))) = \im ({K_0({i}_+)}) \oplus \ZZ \cong \coker (\delta_+)  \oplus \ZZ \cong \ZZ^l,
$$
where the last isomorphism is a consequence of the isomorphism
$$
\ZZ^{l} / \im (\delta_+) \to \ZZ^{l-1} , \qquad [(n_1,n_2,\ldots, n_l)]\mapsto (n_2-n_1,n_3-n_1,\ldots, n_l-n_1).
$$

In particular, for $l=1$, the quantum disc algebra $C(D_q):= C(\RR\PP_q(1;+))$ is isomorphic to the Toeplitz algebra, and the above calculation returns $K$-groups for the Toeplitz algebra; see \cite[Example~9.4.4]{RorLar:int}. 

In fact, the above derivation of $K$-groups can be used for obtaining better insight into the structure of $C(\RR\PP_q(l;+))$. Denote by $\Tt_r$ the $C^*$-algebra generated by the unilateral shift $U_r^*: e_n^r \mapsto e_{n+1}^r$. This is a concrete realisation of the Toeplitz algebra. $\Tt_r$ contains compact operators $\Kk_r$ and projects on $C(S^1)$ by the symbol map $\sigma_r: U_r\mapsto u$, where $u$ is the unitary generator of $C(S^1)$ (which we shall identify with $p_+(c_+)$). More precisely, there is a short exact sequence
$$
\xymatrix{0 \ar[r] & \cK_r \ar[r] & \Tt_r \ar[r]^-{\sigma_r} & C(S^1) \ar[r] & 0.}
$$
Since
$$
(\pi_r(c_+) -U_r)e_n^r = \left(\prod_{m=1}^{l}\left(1 - q^{2(ln+r-m)}\right)^{1/2} - 1\right)e_{n-1}^r,
$$
and $\prod_{m=1}^{l}\left(1 - q^{2(ln+r-m)}\right)^{1/2} - 1$ converges to $0$ as $n$ tends to infinity, the operator $\pi_r(c_+) -U_r$ is compact. Consequently,
$
\sigma_r(\pi_r(c_+)) = u.
$
Furthermore, $\pi_r(a)$ is a compact operator hence, for all $x\in C(\RR\PP_q(l;+))$, and  $r,s = 1,2,\ldots , l$,
$$
\sigma_r(\pi_r(x)) = \sigma_s(\pi_s(x)).
$$
This means that
\begin{eqnarray*}
\bigoplus_{r=1}^{l}\cK_r \subseteq \pi(C(\RR\PP_q(l;+))) \!\!\!&\subseteq&\!\!\! \Tt_1\oplus_{\sigma}\Tt_2\oplus_{\sigma}\cdots \oplus_{\sigma}\Tt_l\\
& :=&\!\!\!  \{(x_1,\ldots, x_l)\in \Tt_1\oplus\Tt_2\oplus\cdots \oplus\Tt_l\; |\; \sigma_r(x_r) = \sigma_s(x_s), \forall r,s\},
\end{eqnarray*}
and, in combination with \eqref{seq}, yields the following commutative diagram with exact rows
$$
\xymatrix{0 \ar[r] & \bigoplus_{r=1}^{l}\cK_r \ar[r]\ar@{=}[d] & \pi\left(C(\RR\PP_q(l;+))\right) \ar[r]\ar@^{(->}[d] & C(S^1)\ar@{=}[d]  \ar[r] & 0\\
0 \ar[r] & \bigoplus_{r=1}^{l}\cK_r \ar[r] & \Tt_1\oplus_{\sigma}\Tt_2\oplus_{\sigma}\cdots \oplus_{\sigma}\Tt_l \ar[r] & C(S^1) \ar[r] & 0.}
$$
Therefore,
$$
\pi\left(C(\RR\PP_q(l;+))\right) \cong \Tt_1\oplus_{\sigma}\Tt_2\oplus_{\sigma}\cdots \oplus_{\sigma}\Tt_l .
$$
Since each of the $\Tt_r$ is isomorphic to the Toeplitz algebra $\Tt\cong C(D_q)$, each of the $\sigma_r$ corresponds to the symbol map $\sigma$,  and  $\pi$ is a faithful representation we conclude that
$$
C(\RR\PP_q(l;+)) \cong \underbrace{\Tt\oplus_{\sigma}\Tt\oplus_{\sigma}\cdots \oplus_{\sigma}\Tt}_{l-\mbox{times}}  \cong \underbrace{C(D_q)\oplus_{\sigma}C(D_q)\oplus_{\sigma}\cdots \oplus_{\sigma}C(D_q)}_{l-\mbox{times}} .
$$
The symbol map $\sigma$ in this last expression can be understood as a projection induced by the inclusion of the classical circle as the boundary of the quantum disc. 

\subsection{$K$-groups of $C(\RR\PP_q(l;-))$.}\label{sec.k.odd}
The derivation of $K$-groups follows the steps outlined in Section~\ref{sec.k.even}. The unitary $p(c_-)$ is now lifted to a coisometry $d_-$ given by  $\pi_r(d_-) = V_r$, $r=1,2,\ldots, r$, where $V_r$ is the shift $V_r(e_0^r) = 0$, $V_r(e_1^r) = 0$, $ V_r(e_n^r) =e_{n-2}^r$, $n\geq 1$. The existence of $d_-$ follows by the observation that
$$
V_r = \pi_r(c_-)\prod_{m=1}^{2l} (1-q^{-2m} \pi_r(a))^{-{1}/2}.
$$
Since $(\id - V_r^*V_r)e_n^r = \delta_{n,0} +\delta_{n,1}$, the index map is an injective function
$$
\delta_- :\ZZ\to \ZZ^l , \qquad m\mapsto (2m,2m,\ldots, 2m).
$$
Consequently, $K_1(C(\RR\PP_q(l;-))) = 0$ and
$$
K_0(C(\RR\PP_q(l;-)))  \cong \coker (\delta_-) \oplus \ZZ \cong \ZZ_2\oplus \ZZ^l,
$$
where the last isomorphism is a consequence of the isomorphism
$$
\ZZ^{l} / \im (\delta_-) \to \ZZ_2\oplus \ZZ^{l-1} , \qquad [(n_1,n_2,\ldots, n_l)]\mapsto ([n_1], n_2-n_1,n_3-n_1,\ldots, n_l-n_1).
$$

In particular, for $l=1$ the above calculation returns $K$-groups for the algebra of continuous functions on the quantum real projective space $C(\RR\PP^2_q)$ calculated in \cite[Theorem~4.8]{HajMat:rea}. By the same analysis as in the preceding section, which is enabled by the existence of a short exact sequence
$$
\xymatrix{0 \ar[r] & \cK \ar[r] & C(\RR\PP^2_q) \ar[r] & C(S^1) \ar[r] & 0,}
$$
(see \cite[p.\ 190]{HajMat:rea}), one finds that 
$$
C(\RR\PP_q(l;+)) \cong \underbrace{C(\RR\PP^2_q)\oplus_{\bar\sigma}C(\RR\PP^2_q)\oplus_{\bar\sigma}\cdots \oplus_{\bar\sigma}C(\RR\PP^2_q)}_{l-\mbox{times}} ,
$$
where the map $\bar\sigma$ is induced from the projection that corresponds to the inclusion of the classical circle as the equator of the quantum sphere $S_q^2$. Explicitly, $\bar\sigma(V_r) = u$, where $u$ is the unitary generator of $C(S^1)$. 

\section*{Acknowledgements.}
I would like to thank the organisers of the {\em Workshop on 
Noncommutative Field Theory and Gravity} (Corfu, September 7--11, 2011)  for support, hospitality and for creating stimulating  environment for exchange of ideas.

\end{document}